\documentclass[12pt]{amsart}
\usepackage{amsmath,amssymb,amsthm}
\usepackage{soul}
\usepackage{hyperref}
\usepackage{latexsym}            
\usepackage{color,colordvi,graphicx}
\numberwithin{equation}{section}

\newtheorem{theorem}{Theorem}[section]
\newtheorem{proposition}[theorem]{Proposition}
\newtheorem{lemma}[theorem]{Lemma}
\newtheorem{corollary}[theorem]{Corollary}

\theoremstyle{remark}

\newtheorem{remark}[theorem]{Remark}

\newcounter{FNC}[page]
\def\fauxfootnote#1{{\addtocounter{FNC}{2}\Magenta{$^\fnsymbol{FNC}$}%
     \let\thefootnote\relax\footnotetext{\Magenta{$^\fnsymbol{FNC}$#1}}}}
\newcommand{\defcolor}[1]{\RoyalBlue{#1}}
\newcommand{\demph}[1]{\defcolor{{\sl #1}}}

\newcommand{\PP}{{\mathbb P}}
\newcommand{\RR}{{\mathbb R}}
\newcommand{\CC}{{\mathbb C}}

\newcommand{\calH}{{\mathcal H}}

\newcommand{\frakL}{{\mathfrak L}}

\newcommand{\tb}{\textcolor{blue}}

\DeclareMathOperator{\Gr}{{\rm Gr}}
\DeclareMathOperator{\Mat}{{\rm Mat}}
\DeclareMathOperator{\Wr}{{\rm Wr}}
\DeclareMathOperator{\GL}{{\it GL}}
\DeclareMathOperator{\Sp}{{\it Sp}}
\DeclareMathOperator{\SO}{{\it SO}}

\DeclareMathOperator{\Span}{{\rm span}}
\DeclareMathOperator{\ev}{{\it ev}}

\title{Injectivity of generalized Wronski maps}

\author{Yanhe Huang}
\address{Yanhe Huang \\
         Department of Mathematics\\
         University of California\\
         Berkeley\\
         CA\ 94720\\
         USA}
\email{yanhe\_huang@berkeley.edu}
\author{Frank Sottile}
\address{Frank Sottile \\
         Department of Mathematics\\
         Texas A\&M University\\
         College Station\\
         Texas \ 77843\\
         USA}
\email{sottile@math.tamu.edu}
\urladdr{\url{http://www.math.tamu.edu/~sottile}}
\author{Igor Zelenko}
\address{Igor Zelenko\\
         Department of Mathematics\\
         Texas A\&M University\\
         College Station\\
         Texas \ 77843\\
         USA}
\email{zelenko@math.tamu.edu}
\urladdr{\url{http://www.math.tamu.edu/~zelenko}}
\thanks{Research of Sottile supported in part by NSF grant DMS-1501370}
\thanks{Research of Zelenko supported in part by NSF grant DMS-1406193}
\thanks{This work is based on the 2016 Master's thesis of Huang}
\begin{document}

\begin{abstract}
 We study linear projections on Pl\"ucker space whose restriction to the Grassmannian is a non-trivial branched
 cover.
 When an automorphism of the Grassmannian preserves the fibers, we show that the Grassmannian is necessarily
 of $m$-dimensional linear subspaces in a  symplectic vector space of dimension $2m$, and the linear map is
 the Lagrangian involution.
 The Wronski map for a self-adjoint linear differential operator and pole placement map for
 symmetric linear systems are natural examples.
\end{abstract}

\maketitle

%
\section*{Introduction}

Some applications of geometry involve maps on Grassmannians that are projections on the ambient
Pl\"ucker space---these generalized Wronski maps include the classical Wronskian in
differential equations and pole placement map in feedback control of linear systems~\cite{Delch}.
When the center of projection is disjoint from the Grassmannian and has maximal dimension, the image is a
projective space of the same dimension as the Grassmannian and the map is a branched cover of degree equal to
the degree of the Grassmannian.
When the center of projection does not have maximal dimension, the image of the
  Grassmannian is a proper subset of projective space.
If in addition the center is general, there is an open subset of the
Grassmannian on which the map is injective.
We consider generalized Wronski maps when the center does not have maximal dimension and yet the map is
a non-trivial branched cover.
This occurs when an automorphism of the Grassmannian preserves each fiber and thus induces the
identity on the image.

Chow~\cite{Chow} classified automorphisms of the Grassmannian of $m$-planes in a complex vector space $V$.
When $2m\neq\dim V$, they are induced by automorphisms of $\PP(V)$, and when $2m=\dim V$ there are
additional automorphisms induced by isomorphisms between $\PP(V)$ and its dual projective space $\PP(V^*)$.
Given a projection with center $Z$ which does not meet the Grassmannian where the automorphism $\varphi$
preserves the fibers of the projection, we show that $V$ is a symplectic vector space of dimension $2m$ with
$\varphi$ the Lagrangian involution $\frakL$, and that $Z$ contains the $(-1)$-eigenspace of $\frakL$.
We also show that any generalized Wronski map of degree 2 has this form.

Such maps arise in nature.
One source is the Wronski map on $m$-dimensional spaces of functions that satisfy a self-adjoint linear
differential equation $L y=0$ of order $2m$.
Another is the pole placement map for a symmetric linear system.
In both cases, the center $Z$ strictly contains the $(-1)$-eigenspace of $\frakL$, and in fact contains all
irreducible summands under the action of the symplectic group, except the one meeting the Grassmannian.
We call such a generalized Wronski map self-adjoint.
This structure (Lagrangian involution preserving the fiber) of the classical Wronski map (when
$Ly=y^{(2m)}$) implies a congruence modulo four on the number of real solutions to certain problems in the
real Schubert Calculus~\cite{HSZII,HSZI}, and was our motivation.

While this classification of projections that induce a non-trivial branched cover on the Grassmannian is
complete when the fibers are preserved by an automorphism of the Grassmannian, we do not know if there are
other such projections not coming from an automorphism of the Grassmannian.
More specifically, we ask the following questions:
Are there any generalized Wronski maps with degree exceeding 2 not arising from an automorphism of the
Grassmannian and whose image is not a projective space?
What is the case when the center of the projection meets the Grassmannian, but the projection still has finite
fibers over an open subset of its image?

In Section~\ref{S:PnG} we prove our main results, Corollary~\ref{C:main} and Theorem~\ref{Th:One},  about
projections on Grassmannians and automorphisms.
Section~\ref{S:Wronskian} shows how self-adjoint linear differential operators are a source of such maps.
In Section~\ref{S:control}, we explain how symmetric linear control systems are another source.

In Sections~\ref{S:Wronskian} and~\ref{S:control}, the projection comes from a curve in a
Grassmannian. 
For a linear control system this is the Hermann-Martin curve~\cite{HM78}.
For a linear differential operator this is  a curve of
osculating spaces  associated  to the operator \cite{shapiro1,W1906}.
These osculaating curves are important in other mathematical topics, such as Sturmian theory of self-adjoint
linear differential equations~\cite{ArnoldSturm,ovsienko}, general linear differential
equations~\cite{shapiro3}, differential geometry of nonlinear differential equations~\cite{doub2},
variational problems~\cite{var}, rank $2$ distributions~\cite{doubzel2}, and single-input nonlinear
control systems~\cite{dzcont}.

We thank Mark Pankov, who pointed us to the work of Chow and the anonymous referee for valuable comments.

\section{Projections and Grassmannians}\label{S:PnG}

We use elementary algebraic geometry as may be found in Harris' vivid book~\cite{Harris}.

\subsection{Projections}\label{SS:projection}

Let \defcolor{$V$} be a finite-dimensional complex vector space and write \defcolor{$\PP(V)$} for the
projective space of one-dimensional linear subspaces of $V$.
For any proper linear subspace $Z$ of $V$ the \demph{projection $\pi_Z$} with
\demph{center} $\PP(Z)$,
 \begin{equation}\label{Eq:projection}
  \pi_Z\ \colon\ \PP(V)\smallsetminus\PP(Z)\ \longrightarrow\ \PP(V/Z)\,,
 \end{equation}
is induced by the quotient map $V\twoheadrightarrow V/Z$.
This projection is only a rational map on $\PP(V)$ as it is not defined for points of $\PP(Z)$.

If $\defcolor{X}\subset\PP(V)$ is an algebraic variety that is disjoint from the center $\PP(Z)$, then we
may restrict $\pi_Z$ to $X$ and obtain a map $\pi_Z\colon X\to \PP(V/Z)$.
Note that $\dim X \leq \dim\PP(V/Z)$.
When the dimensions are equal, so that $\PP(Z)$ has the maximal dimension of a linear subspace disjoint from
$X$, then the projection is surjective, realizing $X$ as a branched cover of $\PP(V/Z)$ of degree equal to the
degree of $X$.
Indeed, a point $p\in\PP(V/Z)$ corresponds to a linear subspace $M$ containing $Z$ as a hyperplane in that
$\dim M=1+\dim Z$.
Then the inverse image of $p$ under $\pi_Z$ in $X$ consists of the points $M\cap X$, which is necessarily
zero-dimensional as $Z\cap X=\emptyset$ and $Z$ is a hyperplane in $M$.
The number of points in $M\cap X$, counted with multiplicity, is the degree of $X$ in $\PP(Z)$.
As the projection map is proper and has finite fibers, it is a finite
morphism~\cite[Part~4, Corollaire~18.12.4]{EGAIV}.

The automorphism group of $\PP(V)$ is the \demph{projective linear group $\Gamma(V)$}.
This is the quotient of the group $\GL(V)$ of invertible linear transformations of $V$ by scalars (which act
trivially on $\PP(V)$).
Thus any automorphism $\varphi$ of $\PP(V)$ is induced by a linear map $\defcolor{\psi}\in\GL(V)$.
This \demph{lift} $\psi$ is well-defined up to multiplication by a scalar.
If $Z\subset V$ is a linear subspace such that $\PP(Z)$ is preserved by $\varphi$, then $\psi$
preserves $Z$ and acts on $V/Z$.
This induces an automorphism $\varphi$ on $\PP(V/Z)$ and the projection map $\pi_Z$ is $\varphi$-equivariant,
\[
    \varphi(\pi_Z(v))\ =\ \pi_Z(\varphi(v))\qquad \mbox{ for }v\in\PP(V)\smallsetminus\PP(Z)\,.
\]

\begin{lemma}\label{L:projection_eigenspace}
 Suppose that $\varphi\in\Gamma(V)$ has finite order and that it preserves the fibers of a projection
 $\pi_Z\colon\PP(V)\smallsetminus\PP(Z)\to\PP(V/Z)$.
 If $\psi$ is any lift of $\varphi$, then $Z$ contains all eigenspaces of
 $\psi$ except one.
\end{lemma}

\begin{proof}
 As $\varphi$ preserves the fibers of $\pi_Z$, it acts as the identity on $\PP(V/Z)$, so that
 $\psi$ acts as scalar multiplication on $V/Z$.
 Since $\varphi$ has finite order, $\psi$ is semisimple, and so
 $V$ is the direct sum of its eigenspaces.
 Since $\psi$ acts as multiplication by a scalar $\lambda$ on $V/Z$, the kernel
 $Z$ must contain all eigenspaces of  $\psi$ with eigenvalue different from $\lambda$.
\end{proof}

In Lemma~\ref{L:projection_eigenspace} we may choose a lift $\psi$ that acts
trivially on the quotient $V/Z$.

\subsection{Automorphisms of Grassmannians}\label{SS:AutGrass}
Let $1\leq m<n=\dim V$ be a positive integer.
The Grassmannian \defcolor{$\Gr_mV$} of $m$-dimensional linear subspaces of $V$ is a smooth projective algebraic
variety of dimension $m(n{-}m)$.
We have $\Gr_1V=\PP(V)$.
The Grassmannian has a natural Pl\"ucker embedding induced by $m$th exterior powers
\[
   \Gr_mV\ \ni\ H\ \longmapsto\ \wedge^mH\ \in\ \PP(\wedge^m V)\,.
\]
(Since $\dim H=m$, the $m$th exterior power $\wedge^mH$ of $H$ is a 1-dimensional subspace of $\wedge^mV$.)
We will always consider $\Gr_mV$ as a subvariety of this \demph{Pl\"ucker space},
\defcolor{$\PP(\wedge^mV)$}.

A linear isomorphism $\psi\in\GL(V)$ acts as an automorphism of both $\Gr_mV$ and $\PP(\wedge^m V)$, and the
Pl\"ucker embedding is $\psi$-equivariant.
Since scalars act trivially on both $\Gr_mV$ and $\PP(\wedge^m V)$, elements of the projective linear group
$\Gamma(V)$ act as automorphisms of $\Gr_mV$, and every such automorphism is the restriction of one on
Pl\"ucker space.

Chow~\cite{Chow} (for a modern exposition, see~\cite{Pankov}) showed that if $2m\neq\dim V$, these are the
only automorphisms of $\Gr_mV$ and  when $2m=\dim V$,  $\Gamma(V)$ has index two in the automorphism group,
forming its identity component.
Elements of the non-identity component are induced by linear isomorphisms $\psi\colon V^*\to V$.
We explain this.

Fix a linear isomorphism $\psi\colon V^*\to V$ and let $m$ be between $1$ and $n-1$, where $n=\dim V$.
For $H\in\Gr_mV$, its annihilator \defcolor{$H^\perp$} is an element of $\Gr_{n-m}(V^*)$, and we define
\defcolor{$H^\psi$} to be $\psi(H^\perp)\in\Gr_{n-m}(V)$.
This map $H\mapsto H^\psi$ is an isomorphism that is induced by an isomorphism of Pl\"ucker space as follows.

Fix a volume form $\Omega\colon\wedge^nV \xrightarrow{\sim}\CC$.
Then there is a pairing $\wedge^mV\times\wedge^{n-m}V\to \CC$ given by
$(\alpha, \beta)\mapsto \Omega(\alpha\wedge\beta)$.
This is nondegenerate, and the identification $(\wedge^{n-m}V)^*=\wedge^{n-m}V^*$ gives the
\demph{Hodge star operator}  $\defcolor{*_\Omega}\colon \wedge^mV \xrightarrow{\sim}\wedge^{n-m}V^*$ which
satisfies $*_\Omega(\wedge^m H)= \wedge^{n-m}H^\perp$.
Composing with the map induced by $\psi$ gives the isomorphism
 \begin{equation}\label{Eq:Psimap}
   \alpha\ \longmapsto\ \alpha^\psi\ \colon\
    \wedge^mV\ \xrightarrow{\ \sim\ }\ \wedge^{n-m}V\,,
 \end{equation}
with the property that $(\wedge^m H)^\psi=\wedge^{n-m}H^\psi$.
The Hodge star operator depends upon the choice of $\Omega$, with any two choices differing by a nonzero
scalar.
This ambiguity is removed by viewing~\eqref{Eq:Psimap} as an isomorphism of Pl\"ucker space,
\[
   \defcolor{(\;)^\psi}\ \colon\ \PP(\wedge^mV)\ \xrightarrow{\ \sim\ }\ \PP(\wedge^{n-m}V)\,,
\]
that extends the map $H\mapsto H^\psi$ on the Grassmannian.
When $2m=\dim V$ this map $(\;)^\psi$ is an automorphism of the Grassmannian $\Gr_mV$ and of Pl\"ucker space
$\PP(\wedge^mV)$.

\begin{proposition}[Chow~\cite{Chow}, Thm.~1]\label{Prop:Chow}
 If $2m\neq\dim V$, then $\Gamma(V)$ is the automorphism group of the Grassmannian $\Gr_mV$.
 When $2m=\dim V$, $\Gamma(V)$ has index two in the automorphism group of the Grassmannian $\Gr_mV$, where the
 elements not from $\Gamma(V)$ have the form $(\;)^\psi$ for some isomorphism $\psi\colon V^*\xrightarrow{\sim}V$.
\end{proposition}

We show explicitly how the square of the map $H\mapsto H^\psi$ for $\psi\colon V^*\xrightarrow{\sim}V$ is
induced by an element of $\GL(V)$.
For $v\in V$, let $\defcolor{\chi(v)}\in V^*$ be the linear map $V\ni u \mapsto v(\psi^{-1}(u))$.
Writing elements of $V$ as linear maps on $V^*$, this is $u(\chi(v))=v(\psi^{-1}(u))$.
Then $\chi\colon V\to V^*$ is an isomorphism as $\psi$ is an isomorphism.
Set $\defcolor{A}:=\psi\circ\chi\in\GL(V)$.

\begin{lemma}\label{L:square}
 For any $H\in\Gr_kV$, $(H^\psi)^\psi=A(H)$.
\end{lemma}

\begin{proof}
 Since $v(\psi^{-1}(u))=u(\chi(v))$, we have that $\chi(H)$ annihilates $H^\psi=\psi(H^\perp)$.
 Then we see that
 $(H^\psi)^\psi =\psi((H^\psi)^\perp)=\psi\circ\chi(H)=A(H)$.
\end{proof}

\subsection{Symplectic vector spaces}\label{SS:symplectic}

A reference for this material is~\cite{GW}, particularly pages 236--7.
An isomorphism $\psi\colon V^*\to V$ is \demph{alternating} if for any $u,v\in V^*$ we have that
\[
    v(\psi(u))\ =\ -u(\psi(v))\,.
\]
Equivalently, if the bilinear form $\defcolor{\langle u,v\rangle} := v(\psi^{-1}(u))$ on $V$ is
nondegenerate and alternating.
Then in Lemma~\ref{L:square} the map $\chi=-\psi^{-1}$ and $A=-I$, 
$V$ has even dimension \defcolor{$2m$}, and the form is represented by an element
$\defcolor{\omega}\in \wedge^2 V^*$ that
is non-degenerate in that $0\neq \wedge^m\omega\in\wedge^{2m}V^*$.
Write $\defcolor{\theta}\in\wedge^2 V$ for the 2-form $\psi(\omega)$.
An even-dimensional vector space with these structures is a \demph{symplectic vector space}, and
$\omega$ is the \demph{symplectic form} on $V$.
The symplectic group \defcolor{$\Sp_\omega(V)$} is the subgroup of $\GL(V)$ preserving the symplectic form
$\omega$.

For  $H\in\Gr_kV$,  $H^\psi\in \Gr_{2m-k}V$ is its annihilator under the symplectic form
\[
   H^\psi\ =\ \{v\in V\mid \langle v,w\rangle = 0\ \mbox{\ for all }w\in H\}\,.
\]
Observe that $(H^\psi)^\psi=H$ as $A=-I$.
For $k\leq m$ a subspace $H\in\Gr_kV$ is \demph{isotropic} if $H\subset H^\psi$.
\demph{Lagrangian} subspaces are isotropic subspaces with maximal dimension $m$.
The symplectic group acts on isotropic subspaces in $\Gr_kV$, which form an orbit.

Let $k\leq m$.
A tensor $v\in\wedge^k V$ is \demph{isotropic} if it lies in a one-dimensional subspace $\wedge^k H$ where
$H\in\Gr_kV$ is isotropic.
Let \defcolor{$\calH(\wedge^k V,\omega)$} be the subspace of $\wedge^k V$ spanned by isotropic tensors.
This is an irreducible representation of $\Sp_\omega(V)$ and the collection of these for
$1\leq k\leq m$ are its fundamental representations.
The exterior product $\wedge^k V$ decomposes as a sum of fundamental representations,
 \begin{equation}\label{Eq:Decomposition}
   \wedge^k V\ =\
     \bigoplus_{p=\max\{0,\lceil (k-m)/2\rceil\}}^{\lfloor k/2\rfloor}
      \wedge^p\theta\wedge  \calH(\wedge^{k-2p}V,\omega)\,.
 \end{equation}

Define $\defcolor{\frakL}\colon\wedge^kV\to\wedge^{2m-k}V$
to be the linear map $(\;)^\psi$ of Subsection~\ref{SS:AutGrass}, using the volume form
$\defcolor{\Omega}:=(-1)^{\binom{m}{2}}\frac{1}{m!}\wedge^m\omega$.
This satisfies $\frakL(\wedge^k H)=\wedge^{2m-k} H^\psi$.
As $\frakL\colon\wedge^mV \to\wedge^m V$ and commutes with the action of  $\Sp_\omega(V)$, it acts by a scalar
on each irreducible summand in~\eqref{Eq:Decomposition} when $m=k$.

\begin{proposition}\label{P:L-action}
 The map $\frakL$ is an involution on $\wedge^m V$ and it acts as multiplication by $(-1)^p$
 on the summand $\wedge^p\theta\wedge  \calH(\wedge^{m-2p}V,\omega)$ in~$\eqref{Eq:Decomposition}$.
\end{proposition}

Call $\frakL$ the \demph{Lagrangian involution}.

\begin{proof}
 We compute $\frakL(v)$ for $v$ lying in one of the summands of~\eqref{Eq:Decomposition}.
 A \demph{Darboux basis} for $V$ provides a normal form for $\omega$.
 Darboux bases always exist; let us fix one for $V$.
 This is a basis $e_1,f_2,\dotsc,e_m,f_m$ with a dual basis $e_1^*,f_1^*,\dotsc,e_m^*,f_m^*$ for $V^*$
 such that
\[
   \omega\ =\ e_1^*\wedge f_1^* +e_2^*\wedge f_2^* + \dotsb + e_m^*\wedge f_m^*\,.
\]
 Then $\psi(e_i^*)=-f_i$, $\psi(f_i^*)=e_i$,
 $\theta = \psi(\omega)=e_1\wedge f_1+\dotsb+e_m\wedge f_m$, and the volume form is
 $\Omega=(-1)^{\binom{m}{2}}\frac{1}{m!}\wedge^m\omega=
       (-1)^{\binom{m}{2}}e_1\wedge f_1\wedge\dotsb\wedge e_m\wedge f_m$.

 Let $0\leq p\leq \lfloor m/2\rfloor$ be an integer, set $\defcolor{[2p]}:=\{1,\dotsc,2p\}$, and let
 \defcolor{$\binom{[2p]}{p}$} be the set of $p$-element subsets of $[2p]$.
 For $I\in\binom{[2p]}{p}$ let
 $\defcolor{(e\wedge f)_I}:=e_{i_1}\wedge f_{i_1}\wedge \dotsb \wedge e_{i_p}\wedge f_{i_p}$ where
 $I=\{i_1,\dotsc,i_p\}$.
 The order of the factors $e_i\wedge f_i$ in this expression does not affect $(e\wedge f)_I$, as doublets
 $e_i\wedge f_i$ commute with all tensors.
 Set $\defcolor{h_I}:=(e\wedge f)_I\wedge e_{2p+1}\wedge\dotsb\wedge e_m$.

 The exterior power $\wedge^m V$ has a basis of tensors $v_1\wedge\dotsb\wedge v_m$ where
 $\{v_1,\dotsc,v_m\}$ is a subset of the basis $\{e_1,\dotsc,e_m,f_1,\dotsc,f_m\}$ for $V$.
 For such a tensor $v$, the map $v\mapsto \Omega(h_I\wedge v)$ is zero unless the components of $v$ are
 elements of the basis that do not appear in $h_I$.
 Thus we may suppose that $v=(e\wedge f)_{I^c}\wedge f_{2p+1}\wedge\dotsb\wedge f_m$, where
 $\defcolor{I^c}:=[2p]\smallsetminus I$ is the complement of $I$.
 Keeping track of the signs induced by permuting factors, we have
 \begin{eqnarray*}
  \Omega(h_I\wedge v)&=& \Omega((e\wedge f)_I\wedge e_{2p+1}\wedge\dotsb\wedge e_m\,\wedge\,
                                (e\wedge f)_{I^c}\wedge f_{2p+1}\wedge\dotsb\wedge f_m)\\
    &=& \Omega((e\wedge f)_I\wedge (e\wedge f)_{I^c}\wedge e_{2p+1}\wedge\dotsb\wedge
           e_m\wedge f_{2p+1}\wedge\dotsb\wedge f_m)\\
    &=& (-1)^{\binom{m}{2}} (-1)^{\binom{m-2p}{2}} \ =\ (-1)^p\,.
 \end{eqnarray*}
 Thus $*_\Omega(h_I)=(-1)^p (e^*\wedge f^*)_{I^c}\wedge f_{2p+1}^*\wedge\dotsb\wedge f_m^*$, and so
 $\frakL(h_I)=\psi(*_\Omega(h_I))=(-1)^p h_{I^c}$.

 Since
\[
   \wedge^p\theta\wedge e_{2p+1}\wedge\dotsb\wedge e_m\ =\
    \tfrac{m!}{(m-p)!} \sum_{I\subset\binom{[2p]}{p}}
      (e\wedge f)_I \wedge e_{2p+1}\wedge\dotsb\wedge e_m\,,
\]
 $\frakL(\wedge^p\theta\wedge e_{2p+1}\wedge\dotsb\wedge e_m)=
   (-1)^p \wedge^p\theta\wedge e_{2p+1}\wedge\dotsb\wedge e_m$, which completes the proof as
 $e_{2p+1}\wedge\dotsb\wedge e_m$ is an isotropic tensor in $\wedge^{m-2p}V$.
\end{proof}

We will call the restriction of a projection on $\wedge^mV$ to the Grassmannian $\Gr_mV$ a
\demph{generalized Wronski map}.

\begin{corollary}\label{C:main}
 Suppose that $V\simeq\CC^{2m}$ is symplectic, $Z\subset\wedge^mV$ contains the
 $(-1)$-eigenspace of the Lagrangian involution $\frakL$, and $\PP(Z)$ is disjoint from the Grassmannian
 $\Gr_mV$.
 Then $\frakL$ acts on the fibers of $\pi_Z$ on $\Gr_mV$, which has even degree over its image.
\end{corollary}

\subsection{Projections commuting with automorphisms}

We prove our main theorem on linear projections of Grassmannians whose fibers are preserved by automorphisms.
This shows that Corollary~\ref{C:main} is the only case in this situation.

\begin{theorem}\label{Th:One}
 Let $Z\subset\wedge^m V$ be a linear subspace with $\PP(Z)$ disjoint from the Grassmannian $\Gr_mV$ and
 $\pi_Z\colon\Gr_mV\to\PP((\wedge^m V)/Z)$ the generalized Wronski map.
 If $\varphi$ is an automorphism of the Grassmannian of finite order at least $2$ that preserves the fibers
 of $\pi_Z$, then $V$ is a symplectic vector space of dimension $2m$, $\varphi$ is the
 Lagrangian involution $\frakL$ on $\Gr_mV$, and $Z$ contains the $(-1)$-eigenspace of $\frakL$ acting on
 $\wedge^m V$.
\end{theorem}

We begin with a lemma.

\begin{lemma}\label{L:meetCenter}
 Let $\psi\in\GL(V)$ be semisimple.
 Then $\Gr_mV$ meets $\PP(Z)$ for $Z$ any eigenspace of $\psi$ acting on $\wedge^m V$.
\end{lemma}

\begin{proof}
 The vector space $V$ has an eigenbasis $e_1,\dotsc,e_n$ ($n=\dim V$) where for each $i$,
 $\psi(e_i)=\lambda_i e_i$ with $\lambda_i$ the corresponding eigenvalue.
 The basis of $\wedge^m V$ of tensors  $e_I:=e_{i_1}\wedge\dotsb\wedge e_{i_m}$ for
 $I=\{i_1,\dotsc,i_m\}\subset\{1,\dotsc,n\}$ is an eigenbasis for $\psi$ acting on $\wedge^m V$.
 Indeed, $\psi(e_I)=\lambda_I e_I$, where  $\lambda_I=\lambda_{i_1}\dotsb\lambda_{i_m}$.
 The lemma follows as $e_I$ spans the image of the $m$-plane spanned by
 $e_{i_1},\dotsc,e_{i_m}$ under the
 Pl\"ucker embedding.
\end{proof}

\begin{proof}[Proof of Theorem~$\ref{Th:One}$]
 By Chow's Theorem~\cite{Chow}, $\varphi$ is the restriction of an automorphism (also written $\varphi$)
 of $\PP(\wedge^m V)$.
 Since $\varphi$ preserves the fibers of $\pi_Z$ on $\Gr_m V$, it fixes its image
 $\pi_Z(\Gr_mV)\subset\PP((\wedge^mV)/Z)$ pointwise.
 As $\Gr_mV$ spans $\PP(\wedge^mV)$, its image spans $\PP((\wedge^mV)/Z)$ and so $\varphi$ fixes
 $\PP((\wedge^mV)/Z)$ pointwise.
 Therefore $\varphi$ preserves the fibers of the projection map
 $\pi_Z\colon \PP(\wedge^mV)\smallsetminus\PP(Z)\to\PP((\wedge^mV)/Z)$.
 By Lemma~\ref{L:projection_eigenspace}, $Z$ contains all eigenspaces except one of any lift
 $\widetilde{\varphi}$ of $\varphi$.
 Since $\varphi$ is not the identity, $\widetilde{\varphi}$ has more than one eigenspace, and so $Z$
 contains at least one eigenspace of $\widetilde{\varphi}$.

 The automorphism $\varphi$ of $\Gr_mV$ has one of two types.
 Either it is induced by a linear automorphism $\psi$ of $V$ or by an isomorphism
 $\psi\colon V^*\to V$ and $2m=\dim V$.
 We show that the first type cannot occur.
 Suppose that $\varphi$ is induced by $\psi\in\GL(V)$.
 As  $\varphi$ has finite order, $\psi$ is semisimple, and by Lemma~\ref{L:meetCenter}, $\Gr_mV$ meets
 every eigenspace of $\psi$ acting on $\wedge^m V$, and therefore $\Gr_mV$ meets $\PP(Z)$, a
 contradiction.

 We are left with the possibility that $2m=\dim V$ and that $\varphi$ is induced by an isomorphism
 $\psi\colon V^*\to V$.
 Since $\varphi$ lies in the non-identity component of the automorphism group of $\Gr_mV$, its square
 $\varphi^2$ lies in the identity component, and is therefore induced by an element of $\GL(V)$.
 Since $\varphi^2$ also preserves the fibers of $\pi_Z$, our previous arguments imply that $\varphi^2$ is
 the identity, and thus $\varphi$ is an involution.

 In particular, this means that $H=(H^\psi)^\psi$ for all $H\in\Gr_mV$.
 By Lemma~\ref{L:square}, if $\chi\colon V\to V^*$ is the isomorphism defined by $v(\psi^{-1}(u))=u(\chi(v))$,
 and $A:=\psi\circ\chi$, then $A(H)=H$ for all $H\in\Gr_mV$.
 This implies that $A$ is a scalar matrix, $A=cI$, for some scalar $c$.
 The computation
\[
   Av(\psi^{-1}(Au))\ =\ Av(\chi(u))\ =\ u(\psi^{-1}(Av))\ =\ u(\chi(v))\ =\ v(\psi^{-1}(u))\,,
\]
 implies that $c^2=1$ and so either $c=1$ or $c=-1$.

 Consider the nondegenerate bilinear form on $V$ defined by $\defcolor{\langle u,v\rangle}:=u(\psi^{-1}(v))$.
 This is symmetric when $c=1$ and alternating when $c=-1$.
 Suppose that $c=-1$.
 Then the map $\varphi$ is the Lagrangian involution $\frakL$.
 By Proposition~\ref{P:L-action} and the decomposition~\eqref{Eq:Decomposition} of $\wedge^mV$ into
 irreducible representations of $\Sp_\omega(V)$, $\frakL$ has two eigenspaces on $\wedge^mV$ with eigenvalues
 $+1$ and $-1$.
 By Lemma~\ref{L:projection_eigenspace}, $Z$ must contain one of them.
 The $+1$ eigenspace contains $\calH(\wedge^m V,\omega)$, which is spanned by isotropic (and even Lagrangian)
 tensors.
 As these are elements of $\Gr_mV$, we deduce that $Z$ contains the $-1$ eigenspace as
 $\PP(Z)\cap\Gr_mV=\emptyset$.

 To complete the proof, assume that $c=1$ so that the form $\langle\;,\;\rangle$ induced by $\psi$ is symmetric.
 The identity component of the subgroup of $\GL(V)$ of linear maps that preserve the form is the
 \demph{special orthogonal group}, \defcolor{$\SO(2m)$}.
 As explained in~\cite{GW} on page 235, under $\SO(2m)$, $\wedge^mV$ decomposes into
 two irreducible summands,
 \begin{equation}\label{Eq:SO2mDecomposition}
   \wedge^mV\ =\ W_{2\varpi_{m-1}}\oplus W_{2\varpi_m}\,,
 \end{equation}
 where $\varpi_{m-1}$ and  $\varpi_m$ are highest weights of the two half-spin representations of
 $\SO(2m)$.
 As in Subsection~\ref{SS:symplectic}, the involution  $(\;)^\psi$ on $\wedge^mV$ commutes with $\SO(2m)$, and
 so the summands in~\eqref{Eq:SO2mDecomposition} are eigenspaces of $(\;)^\psi$.
 Since one is the $(+1)$-eigenspace and the other the $(-1)$-eigenspace, $Z$ must contain one summand.
 Since each summand is spanned by isotropic vectors, $\PP(Z)$ meets the Grassmannian $\Gr_mV$, a
 contradiction.
\end{proof}

\begin{corollary}\label{C:degreeTwo}
 Let $Z\subset\wedge^m V$ be a linear subspace with $\PP(Z)$ disjoint from the Grassmannian $\Gr_mV$
 with $\pi_Z\colon\Gr_mV\to\PP((\wedge^m V)/Z)$ the corresponding generalized Wronski map.
 If the map $\pi_Z$ on $\Gr_mV$ has degree $2$ with finite fibers, then $V$ is a symplectic vector space of
 dimension $2m$ with $Z$ containing the $(-1)$-eigenspace of the Lagrangian involution $\frakL$ on
 $\wedge^mV$, and $\frakL$ acts on each each fiber of $\pi_Z$ on $\Gr_mV$.
\end{corollary}

\begin{proof}
 Since $\pi_Z\colon\Gr_mV\to\PP((\wedge^m V)/Z)$ is proper, each fiber
 consists of one or two points.
 Interchanging
 the points when there are two is a global analytic involution on $\Gr_mV$.
 By Chow's Theorem XV in~\cite{Chow}, any analytic automorphism of $\Gr_mV$ is algebraic of
 the form given in Proposition~\ref{Prop:Chow}, so that we are in the situation
 of Theorem~\ref{Th:One}.
\end{proof}

\section{Wronski map for self-adjoint differential operators}\label{S:Wronskian}

Let $V$ be a finite-dimensional vector space of sufficiently differentiable complex functions on an
open interval $I\subset\RR$.
Given linearly independent functions $f_1,\dotsc,f_m\in V$, their \demph{Wronskian} is the function on $I$
defined by the determinant
 \begin{equation}\label{Eq:Wronskian}
   \defcolor{\Wr(f_1,\dotsc,f_m)}\ :=\  \det
    \left(\begin{matrix} f_1(t) & f'_1(t) & \dotsb & f^{(m-1)}_1(t)\\
                         \vdots & \vdots  & \dotsb & \vdots \\
                         f_m(t) & f'_m(t) & \dotsb & f^{(m-1)}_m(t)
          \end{matrix}\right)\ .
 \end{equation}
Up to a scalar, this depends only upon the linear span of the functions $f_1,\dotsc,f_m$.
If $V$ is a space such that no such Wronskian vanishes identically (for example, if $V$ consists of analytic
functions),
then the Wronskian is a map from the Grassmannian $\Gr_mV$ to a projective space of functions.
We are interested in cases when the Wronskian realizes $\Gr_mV$ as a non-trivial branched cover of its
image.

When $V=\defcolor{\CC_{n-1}[t]}$ is the space of univariate polynomials of degree at most $n{-}1$, the
Wronskian is such a map from $\Gr_m\CC_{n-1}[t]$ to $\PP(\CC_{m(n-m)}[t])$ of degree
\[
      (m(n{-}m))!\cdot\frac{1!2!\dotsb(n{-}m{-}1)!}{m!(m{+}1)!\dotsb(n{-}1)!}\,,
\]
 the degree of the Grassmannian in Pl\"ucker space~\cite{Sch1886c}.
 This Wronski map, while classical, has been essential in the theory of limit linear series~\cite{EH83,EH87}
and
 in the resolution of the Shapiro conjecture~\cite{EG02,MTV,FRSC}.
 It is a linear projection on Pl\"ucker space applied to the Grassmannian arising from the
 linear differential operator $Ly = y^{(n)}=0$.

 To begin to explain this, let $L$ be a linear differential operator of order $n$ on $I$,
 \begin{equation}\label{Eq:LDO}
   L y\ =\ y^{(n)} + a_{n-1} y^{(n-1)} + \dotsb + a_0 y\,,
 \end{equation}
 where $a_0,\dotsc,a_{n-1}$ are complex-valued smooth functions on $I$.
 Define \defcolor{$V_L$} to be the complex vector space of solutions to the homogeneous differential equation
 $Ly=0$.

\begin{proposition}
 An $n$-dimensional space $V$ of functions on an interval $I$ is the space $V_L$ of solutions of the homogeneous equation corresponding to a linear
 differential operator $L$ as in~$\eqref{Eq:LDO}$ if and only if $\Wr(V)$ is a nowhere-vanishing function on
 $I$.
\end{proposition}

\begin{proof}
 For sufficiency, let $f_1,\dotsc,f_n$ be any basis for $V$, then
 $y\in V$ if and only if
 \begin{equation}
  \label{Eq:LDO_Wronskian}
   Ly\ =\ \frac{(-1)^n}{\Wr(f_1,\dotsc,f_n)}  \cdot  \det
    \left(\begin{matrix} y & y' & \dotsb & y^{(n)}\\
                         f_1 & f'_1 & \dotsb & f^{(n)}_1\\
                         \vdots & \vdots  & \ddots & \vdots \\
                         f_n & f'_n & \dotsb & f^{(n)}_n
          \end{matrix}\right)\ =\ 0\,,
 \end{equation}
 and necessity is provided by the classical Abel Theorem.
\end{proof}

\subsection{The Wronski map is a projection}\label{SS:Wronski_projection}

We henceforth assume that $V=V_L$ is the space of functions associated to a linear differential
operator $L$~\eqref{Eq:LDO} of order $n$.
Equivalently, that the Wronskian $\Wr(V)$ is a nowhere-vanishing function on $I$.

To any linear differential operator $L$ one can assign a curve in a projective space and the
corresponding osculating  curves in Grassmannians.
This  is well known (see the classical book  of
Wilczynski~\cite[p.~51]{W1906} or \cite{shapiro1} or \cite[\S\S~2.2]{ovstab} for modern expositions).

For $t\in I$ the \demph{evaluation map}
\[
   \defcolor{\ev}(t)\ =\ \ev_L(t)\ \colon\ V\ \longrightarrow\ \CC
    \qquad  f\ \longmapsto\ f(t)
\]
 is an element of $V^*$.
 Then $t\mapsto \ev(t)$ is a smooth map $\ev\colon I\to V^*$.
 For each $i=0,1,\dotsc,n{-}1$ and $t\in I$, let
\[
   \defcolor{E^{(i)}(t)}\ =\ E^{(i)}_L(t)\ :=\
    \Span\{\ev(t),\ev'(t),\dotsc,\ev^{(i)}(t)\}\,,
\]
 be the $i$th osculating space to the curve $\ev(I)$ at $\ev(t)$.

\begin{lemma}\label{L:dimE(i)}
 For $t\in I$ and $0\leq i\leq n{-}1$, the osculating space $E^{(i)}(t)$ has dimension $i{+}1$.
\end{lemma}

A curve $\gamma\colon I\to V^*$ whose $i$th osculating spaces have dimension $i{+}1$ in $\PP(V^*)$ for every $i$
and $t\in I$ is \demph{convex}.
Lemma~\ref{L:dimE(i)} implies that $\ev$ is convex and thus for every $i$, $E^{(i)}\colon I\to \Gr_{i+1}V^*$
is a curve in the Grassmannian.

\begin{proof}
 Let $f_1,\dotsc,f_n$ be a basis for $V$ with dual basis $f_1^*,\dotsc,f_n^*$.
 Observe that for $t\in I$, we have $\ev(t)=f_1(t)f_1^*+\dotsb+f_n(t) f_n^*$.
 Consequently, $\ev^{(i)}(t)=f_1^{(i)}(t)f_1^*+\dotsb+f_n^{(i)}(t) f_n^*$.
 As $\Wr(V)(t)\neq 0$, the $n$ column vectors in~\eqref{Eq:Wronskian} (where $m=n$) are linearly
 independent.
 But these are $\ev(t),\ev'(t),\dotsc,\ev^{(n-1)}(t)$.
 Thus the first $i{+}1$ are linearly independent, which implies that  $E^{(i)}(t)$ has dimension
 $i{+}1$.
\end{proof}

We observe that if $f\in V$, $t\in I$ and $i=0,\dotsc,n{-}1$, then $f^{(i)}(t)$ is obtained by evaluating the
linear function $f\in V$ on the vector $\ev^{(i)}(t)\in V^*$.

Let $\defcolor{U}\subset \wedge^mV^*$ be the linear span of the one-dimensional spaces
$\wedge^m E^{(m-1)}(t)$ for $t\in I$.
A linear form $\lambda$ on $U$ defines a function on $I$ by
\[
    \lambda\ \colon\ t\ \longmapsto\ \lambda
     (\ev(t)\wedge \ev'(t)\wedge \dotsb\wedge \ev^{(m-1)}(t))\,.
\]
This identifies the dual space $U^*$ with a space of functions on $I$ as the function $\lambda(t)$ is
identically zero only if $\lambda=0$.
Set $\defcolor{Z}:= U^\perp\subset\wedge^m V$, the annihilator of $U$, so that the quotient
$(\wedge^m V)/Z$ is identified with $U^*$ and thus with this space of functions.

\begin{proposition}
 \label{P:Wronksi_projection}
 With this identification of their codomains, the Wronski map on $\Gr_mV$ equals the projection map $\pi_Z$.
\end{proposition}

This justifies our terminology, that a projection map restricted to the Grassmannian is a generalized Wronski
map.

\begin{proof}
 Let $f_1,\dotsc,f_m\in V$ be linearly independent.
 For $t\in I$, consider the composition
 \begin{equation}\label{Eq:Wronski_Matrix}
   \CC^m\ \longrightarrow\ V^*\ \longrightarrow\ \CC^m\,,
 \end{equation}
 where the first map sends the standard basis element $e_i\in\CC^m$ to $\ev^{(i-1)}(t)$, and the second
 is given by the $m$ linear functions $f_1,\dotsc,f_m$ on $V^*$.
 Expressing this composition as a matrix gives $(f^{(j-1)}_i(t))_{i,j=1}^m$, the matrix of the
 Wronskian~\eqref{Eq:Wronskian}.

 Taking $m$th exterior powers gives the composition
\[
   \CC\ =\ \wedge^m\CC^m\ \longrightarrow\
             \wedge^mV^*\ \longrightarrow\ \wedge^m\CC^m\ =\ \CC\,,
\]
which is multiplication by $\Wr(f_1,\dotsc,f_m)(t)$.
 The first map sends the generator $e_1\wedge\dotsb\wedge e_m$ to
 $\ev(t)\wedge \ev'(t)\wedge \dotsb\wedge \ev^{(m-1)}(t)$ and the second is the linear form on
 $\wedge^m V^*$ given by $f_1\wedge\dotsb\wedge f_m$.
 This identifies the Wronskian with the function $f_1\wedge\dotsb\wedge f_m$ in $U^*$.
\end{proof}

%
\subsection{ Self-dual curves in projective space and self-adjoint differential operators}

We describe the relation between duality of linear differential operators and the
corresponding curves in projective spaces.
This can be found in the classical text~\cite{W1906}.
Details are  also in  any of the modern
sources~\cite{Arnold, ovstab, ovsienko}.

Two curves $\gamma\colon I\to\PP(V)$ and $\widetilde{\gamma}\colon I\to\PP(\widetilde{V})$ are \demph{equivalent}
if there exists a projective isomorphism $\varphi\colon\PP(V)\to\PP(\widetilde{V})$ such that for all $t\in I$,
$\varphi\gamma(t)=\widetilde{\gamma}(t)$.
Two linear differential operators $L$, $\widetilde{L}$ on $I$ with leading coefficient $1$ are \demph{equivalent} if there exists a
nowhere-vanishing function $\mu$ on $I$ such that for all smooth functions $y$,
$\mu \widetilde{L}y=L(\mu y)$.
Since
\[
   (\mu y)^{(n)} + a_{n-1}(\mu y)^{(n-1)} \ =\
   \mu  y^{(n)} + (\mu a_{n-1}+n\mu')y^{(n-1)} + \mbox{lower order terms in $y$} \,,
\]
there is a unique operator equivalent to $L$ whose coefficient of $y^{(n-1)}$ vanishes.

\begin{remark}
\label{equivoprem}
Two linear differential operators $L$ and $\widetilde{L}$ on $I$ are equivalent if  and only if there exists a
nowhere-vanishing function $\mu$ such that $y\in V_L$ if and only if $\mu y \in V_{\widetilde L}$.
\end{remark}

\begin{lemma}
 Let $L,\widetilde{L}$ be linear differential operators on $I$ of order $n$ with
 $\ev\colon I\to\PP(V^*_L)$ and $\widetilde{\ev}\colon I\to\PP(V^*_{\widetilde{L}})$, their corresponding
 evaluation curves.
 Then $L$ is equivalent to $\widetilde{L}$ if and only if $\ev$ is equivalent to $\widetilde{\ev}$.
\end{lemma}

\begin{proof}
 Suppose that $L$ is equivalent to $\widetilde{L}$ and $\mu$ is the nonvanishing function on $I$ such that
 $\mu \widetilde{L}y=L(\mu y)$ for $y$ a function on $I$.
 Then $y\mapsto\mu y$ defines a linear isomorphism $\mu\colon V_{\widetilde{L}}\to V_L$.
 Let $\mu^*\colon V^*_L\to V^*_{\widetilde{L}}$ be the dual map.
 For $t\in I$ and $y\in V_{\widetilde{L}}$, we have
\[
   (\mu^*\!\ev(t))(y)\ =\ \ev(t)(\mu y)\ =\
   \mu(t)\cdot y(t)\ =\ \mu(t)\cdot \widetilde{ev}(t) (y)\,.
\]
 Thus $\mu^*\!\ev(t)$ and $\widetilde{\ev}(t)$ are proportional, which shows that the corresponding curves in
 $\PP(V_L)$ and $\PP(V_{\widetilde{L}})$ are equivalent.

 Suppose that the projective curves $\ev$ and $\widetilde{\ev}$ are equivalent, and let
 $\varphi\colon\PP(V^*_L)\to\PP(V^*_{\widetilde{L}})$ be the projective isomorphism such that
 $\varphi(\ev)=\widetilde{\ev}$.
 Let $\psi\colon V^*_L\to V^*_{\widetilde{L}}$ be a lift of $\varphi$.
 For each $t\in I$ the linear maps $\psi(\ev(t))$ and $\widetilde{\ev}(t)$ on $V_{\widetilde{L}}$ are
 proportional in that $\psi(\ev(t))=\mu(t)\cdot\widetilde{\ev}(t)$.
 Then $\mu$ is smooth and nowhere-vanishing on $I$.
 Let $\psi^*\colon V_{\widetilde{L}}\to V_L$ be the map dual to $\psi$.
 For $y\in V_{\widetilde{L}}$, we have
\[
   \psi^*(y)(t)\ =\ \ev(t)(\psi^*(y))\ =\ \psi(\ev(t))(y)
   \ =\ \mu(t)\cdot\widetilde{\ev}(t)(y)\ =\ \mu(t)\cdot y(t)\,,
\]
 so that $\psi^*(y)=\mu y$.
 By Remark \ref{equivoprem} $L$ is equivalent to $\widetilde{L}$.
\end{proof}

%
Setting $a_n=1$ in the definition~\eqref{Eq:LDO} of a linear differential operator $L$ of order $n$,
its (\demph{formal}) \demph{adjoint $L^*$} is
\[
   L^* y\ :=\ \sum_{i=0}^n (-1)^i (a_i y)^{(i)}\,.
\]
If $L=L^*$ then $L$ is (formally) \demph{self-adjoint}.
This implies that $n=2m$ is even.
When $n$ is odd, the corresponding notion is anti self-adjoint, that $L^*=-L$.
In either case, $a_{n-1}=0$.
At most one operator in an equivalence class is self-adjoint/anti self-adjoint.

Given a convex curve $\gamma\colon I\to V^*$, its dual curve $\gamma^*\colon I\to V$ is defined by setting
$\gamma^*(t)$ to be the $(n{-}2)$nd osculating space to $\gamma$ at $\gamma(t)$.
More specifically, set $\gamma^*(t)=\gamma(t)\wedge \gamma'(t)\wedge\dotsb\wedge\gamma^{(n-2)}(t)$, and then use
an identification of $V$ with $\wedge^{n-1}V^*$.
While this only defines $\gamma^*$ up to a scalar function $\mu(t)$ in $V$, it is well-defined as a curve in
$\PP(V)$.
Observe that $\gamma$ is convex if and only if $\gamma^*$ is convex.
The curve $\gamma$ is \demph{self-dual} if it is equivalent to its dual.
The following has appeared in~\cite[p.~55]{W1906}.
A modern exposition is in \cite[Th.~2.2.6]{ovstab}, and comments concluding Section 2.2 in {\it loc.~cit}.

\begin{proposition}
 Let $\ev_L\colon I\to V_L$ be the curve associated to a linear differential operator $L$.
 Then its dual curve $(\ev_L)^*$ is equivalent to the curve associated to the adjoint operator $\tb{(-1)^n} L^*$.
 An operator is equivalent to a self-adjoint/anti self-adjoint operator $L$ if and only if its curve $\ev_L$
 is self-dual.
\end{proposition}

If a curve $\gamma\colon I\to V^*$ is equivalent to its dual curve $\gamma^*$, there is a linear transformation
$\defcolor{\psi}\colon V^*\to V$ such that
for $t\in I$,
\[
    \gamma^*(t)\ =\ \psi(\gamma(t))\,.
\]
If $\gamma$ is convex, then $\psi$ is an isomorphism.
The following has appeared in~\cite[Rem.~2.2.8]{ovstab}.

\begin{proposition}
 The map $\psi$ is skew-symmetric if $n$ is even and symmetric if $n$ is odd.
\end{proposition}

When $n=2m$ is even, $\psi$ endows $V$ with a symplectic structure.
If $L$ is a self-adjoint linear differential operator, then this is the canonical symplectic structure on
$V_L$.
Write $\defcolor{\omega_L}\in \wedge^2 V^*_L$ for the symplectic form and \defcolor{$\frakL_L$} for the
corresponding Lagrangian involution.
The following is found in~\cite[Lem.~2]{ovsienko}. 

\begin{proposition}
 \label{P:Kwessi_lagrangian}
 Suppose that $L$ has even order $2m$.
 Then $L$ is equivalent to its adjoint $L^*$ if and only if the $(m{-}1)$st osculating
 space $E_L^{(m-1)}(t)$ is Lagrangian.
\end{proposition}

\subsection{The Wronski map of a self-adjoint operator}
\label{SS:SA_Wronski}

Let $L$ be a linear differential operator of even degree $2m$ that is equivalent to its adjoint.
As in Subsection~\ref{SS:Wronski_projection}, let $\defcolor{U}\subset\wedge^m V_L^*$ be the span of the tensors
$\ev(t)\wedge\dotsb\wedge \ev^{(m-1)}(t)$ for $t\in I$.
By Proposition~\ref{P:Kwessi_lagrangian}, these are Lagrangian, so that
$U\subset\calH(\wedge^m V^*_L,\omega^*_L)$.
As the decomposition~\eqref{Eq:Decomposition} for $k=2m$ is preserved by duality, we have
 \begin{equation}\label{Eq:SAsubset}
   Z\ =\ U^\perp\ \supset\
     \bigoplus_{p=1}^{\lfloor m/2\rfloor}
      \wedge^p\theta\wedge  \calH(\wedge^{m-2p}V_L,\omega_L)\ =\
      \theta\wedge(\wedge^{m-2}V)\,.
 \end{equation}
(Note that $\theta=\omega^*_L$.)
In particular, $Z$ properly contains the $(-1)$-eigenspace of the Lagrangian involution $\frakL_L$ when
$m\geq 4$.

\begin{theorem}\label{Th:LDO}
 Let $L$ be a linear differential operator of even degree $2m$ that is equivalent to its adjoint.
 Then the Wronski map on $\Gr_m V_L$ has even degree and the space of functions spanned by
 Wronskians of $m$ solutions of $L$ has dimension at most $\binom{2m}{m}-\binom{2m-2}{m-2}$.
\end{theorem}

In particular, such Wronski maps provide examples of linear projections on Grassmannians that are non-trivial
branched covers of their images.

\begin{proof}
 By Proposition~\ref{P:Wronksi_projection}, the Wronski map is the projection $\pi_Z$ with center $Z$.
 Then Corollary~$\ref{C:main}$ implies the statement about the degree of the Wronski map.
 The statement about the dimension follows as
 $\dim\calH(\wedge^mV_L,\omega_L)=\binom{2m}{m}-\binom{2m-2}{m-2}$.
\end{proof}

\section{Self-adjoint projections and symmetric linear systems}\label{S:control}

Let $\pi_Z\colon\Gr_mV\to\PP((\wedge^m V)/Z)$ be a linear projection on the Grassmannian with center $\PP(Z)$ as
in Section~\ref{S:PnG}.
By Corollary~\ref{C:main}, if $V$ has a symplectic structure given by a form $\omega\in\wedge^2 V^*$ with dual
form $\theta\in\wedge^2 V$, and $Z$ contains the $(-1)$-eigenspace
\[
     \bigoplus_{\substack{p\mbox{\scriptsize\ odd}\\ 1\leq p\leq  m/2} }
      \wedge^p\theta\wedge  \calH(\wedge^{m-2p}V,\omega)
\]
 of the Lagrangian involution $\frakL$, then $\pi_Z$ is a branched cover of even degree over its image.
 In Subsection~\ref{SS:SA_Wronski}, we saw that the Wronski map $\pi_Z$ of a self-adjoint linear differential
 operator $L$ satisfies a stronger property, that the center $Z$ contains all terms of the
 sum~\eqref{Eq:SAsubset}, which is the subspace $\theta\wedge(\wedge^{m-2}V)$.
 A generalized Wronski map $\pi_Z$ is \demph{self-adjoint} if its center $Z$ contains
 $\theta\wedge(\wedge^{m-2}V)$ for some symplectic structure on $V$.
 We show that the pole placement map is self-adjoint in this sense for
 symmetric linear systems of sufficiently high McMillan degree.

\subsection{Pole placement for constant state-space feedback}\label{SS:PP}
 For more on linear systems theory, see~\cite{Delch}.
 A state-space realization of a (strictly proper) $m$-input $p$-output linear system is a triple
 $\Sigma=(A,B,C)$ of matrices of sizes $N\times N$, $N\times m$, and $p\times N$, which defines a system of first
 order constant coefficient linear differential equations,
 \begin{equation}\label{Eq:State_Space}
   \dot{x}\ =\ Ax + Bu\qquad\mbox{and}\qquad y\ =\ Cx\,,
 \end{equation}
 where $x\in\CC^N$, $y\in\CC^p$, and $u\in\CC^m$ are functions of $t\in\CC$ (and
 $\defcolor{\dot{x}}=\frac{d}{dt}x$).
 Applying Laplace transform and assuming $x(0)=0$, we eliminate to obtain
\[
   \widehat{y}(s)\ = \ C(sI-A)^{-1}B\, \widehat{u}(s)
    \ =\ G(s)\, \widehat{u}(s)\,,
\]
 where \defcolor{$\widehat{\ }$} indicates Laplace transform and
 $\defcolor{G(s)}:=C(sI-A)^{-1}B$ is the \demph{transfer function} of~\eqref{Eq:State_Space}.
 This $p\times m$ matrix of rational functions has poles at the eigenvalues of $A$.

 A linear system may be controlled with output feedback, setting
 $u=Ky$, where $K$ is a constant $m\times p$ matrix.
 Substitution in~\eqref{Eq:State_Space} and elimination gives the closed loop system,
\[
   \dot{x}\ =\ (A + BKC) x\,,
\]
 whose transfer function has poles at the zeroes of the characteristic polynomial
 \begin{equation}\label{Eq:PPP}
   \defcolor{P_\Sigma}(K)\ =\ P_\Sigma\ :=\
    \det( sI \ -\ (A+BKC))\,.
 \end{equation}

 The map $K\to P_\Sigma(K)$ is called the \demph{pole placement map}.
 Given a system~\eqref{Eq:State_Space} with state space realization $\Sigma$ and poles
 $\defcolor{z}=\{z_1,\dotsc,z_N\}\subset\CC$, the pole placement problem asks for a matrix $K$ such that
 $P_\Sigma(K)$ vanishes at the points of $z$.
 This is only possible for general $z$ if $N\leq mp$~\cite{Byrnes}.
 We are interested when $N\geq mp$ and the pole placement map is a non-trivial branched cover of its
 image.

Using the injection $\Mat_{m\times p}\CC\to \Gr_p\CC^{m+p}$ where $K$ is sent to the
column space of the matrix $(\begin{smallmatrix}K\\I_p\end{smallmatrix})$, standard manipulations
show that the pole placement map is a linear projection on the Grassmannian $\Gr_p\CC^{m+p}$, of a form similar
to the Wronski map of Subsection~\ref{SS:Wronski_projection}.
For this, the map that sends $s\in\PP^1$ to the column space of
$(\begin{smallmatrix}I_m\\ G(s)\end{smallmatrix})$
defines the \demph{Hermann-Martin curve} $\defcolor{\gamma_\Sigma}\colon\PP^1\to\Gr_m\CC^{m+p}$~\cite{HM78}.
Its degree is the McMillan degree of the system, which is the minimal number $N$ in a state-space
realization giving the transfer function $G(s)$.
Such a minimal representation is observable and controllable~\cite{Delch}.

If $U\subset\wedge^m\Gr_m\CC^{m+p}$ is the linear span of the curve $\gamma(\PP^1)$ as in
Subsection~\ref{SS:Wronski_projection}, and $Z:=U^\perp$ is its annihilator in $\wedge^p\CC^{m+p}$, then the
pole placement map is the generalized Wronski map $\pi_Z$, and we may identify the quotient
$U^*=(\wedge^p\CC^{m+p})/Z$ as the space of polynomials of degree at most $N$.
The pole placement map is \demph{proper} if $\PP(Z)$ is disjoint from the Grassmannian $\Gr_p\CC^{m+p}$.
This terminology is not standard in systems theory.

Two triples  $\Sigma=(A,B,C)$ and $\widetilde{\Sigma}=(\widetilde{A},\widetilde{B},\widetilde{C})$
are \demph{state feedback equivalent} if
 \begin{equation}\label{Eq:SFE}
  \widetilde{A}\ =\ L^{-1}(A+BQT^{-1}C)L\,,\
  \widetilde{B}\ =\ L^{-1}BW\,,\quad\mbox{and}\quad
  \widetilde{C}\ =\ T^{-1}CL\,.
 \end{equation}
for matrices $L,W,T,Q$.
The corresponding transfer functions satisfy
\[
   \widetilde{G}(s)\ =\
    T^{-1} G(s) ( I\ -\ QT^{-1}G(s))^{-1} W\,.
\]

The following is standard.

\begin{proposition}
 State feedback equivalent realizations have equivalent Hermann-Martin curves, where the equivalence is
 induced by an element of $\GL(\CC^{m+p})$.
\end{proposition}

\subsection{Symmetric linear systems}
A linear system is \demph{symmetric} if $p=m$ and the transfer function is symmetric, $G(s)^T=G(s)$.
Symmetric linear systems have symmetric state space realizations~\cite{Fuhr}, where $A^T=A$ and $C^T=B$.
If $G(s)$ is symmetric, then there is a symplectic structure $\omega$ on $\CC^{2m}$ such that the
Hermann-Martin curve in $\Gr_m\CC^{2m}$ lies in the Lagrangian Grassmannian,
$\defcolor{L(m)}=\Gr_mV\cap\calH(\wedge^m\CC^{2m},\omega)$~\cite{HRW,HS}, and vice-versa: if the
Hermann-Martin curve of
a linear system lies in the Lagrangian Grassmannian for a symplectic structure, then that linear system is
state feedback equivalent to a symmetric linear system.
By the discussion of Subsection~\ref{SS:PP} and the same reasoning as Theorem~\ref{Th:LDO}, we deduce the
following.

\begin{theorem}\label{Th:SFC}
 If a controllable and observable linear system is state feedback equivalent to a symmetric system and the
 pole placement map is proper, then it is a self-adjoint generalized Wronski map.
\end{theorem}

The pole placement map for symmetric systems therefore has even degree.

\begin{corollary}
 Given a general feedback law $K$ for a symmetric linear system $\Sigma$, there are an odd number of other
 feedback laws $K_2,\dotsc,K_{2r}$ such that $K,K_2,\dotsc,K_{2r}$ all have the same image
 under the pole placement map.
\end{corollary}

We also have a converse to Theorem~\ref{Th:SFC}.

\begin{corollary}
 If the pole placement map is a self-adjoint generalized Wronski map, then the given system is state feedback
 equivalent to a symmetric system.
\end{corollary}

\begin{proof}
 If the pole placement map is self-adjoint, then the Hermann-Martin curve lies in the Lagrangian
 Grassmannian for some symplectic form.
 But this implies that the original system was state feedback equivalent to a symmetric linear system.
\end{proof}


\providecommand{\bysame}{\leavevmode\hbox to3em{\hrulefill}\thinspace}
\providecommand{\MR}{\relax\ifhmode\unskip\space\fi MR }
\providecommand{\MRhref}[2]{%
  \href{http://www.ams.org/mathscinet-getitem?mr=#1}{#2}
}
\providecommand{\href}[2]{#2}

\end{document}